\documentclass{amsart}
\usepackage{color}
\usepackage{enumerate}

\newtheorem{thm}{Theorem}
\newtheorem{prop}[thm]{Proposition}
\newtheorem{cor}[thm]{Corollary}
\newtheorem{lem}[thm]{Lemma}

\begin{document}

\title{Optimality of Cwikel-Solomyak estimates in Orlicz spaces}

\author[]{F. Sukochev}
\address{School of Mathematics and Statistics, University of New South Wales, Kensington,  2052, Australia}
\email{f.sukochev@unsw.edu.au}

\author[]{D. Zanin}
\address{School of Mathematics and Statistics, University of New South Wales, Kensington,  2052, Australia}
\email{d.zanin@unsw.edu.au}

\begin{abstract}
We discuss the optimality of Cwikel-Solomyak estimates for the uniform operator norm and   establish optimality of M.Z Solomyak's results \cite{Solomyak1995} within the class of Orlicz spaces. Our methods are based on finding the optimal version of the Sobolev embedding theorem.
\end{abstract}

\dedicatory{Dedicated to the memory of M.Z. Solomyak}

\maketitle

\section{Introduction} In this text we establish three main results.

Firstly, we establish a distributional version of Sobolev inequality written in terms of Hardy-Littlewood submajorization (see Theorem \ref{dist sobolev critical}). We demonstrate that our version of this inequality is optimal in the sense that it cannot be improved in terms of the distribution function (see Proposition \ref{estimate from below lemma}).

Secondly, we show that our version of Sobolev inequality yields the optimal Sobolev embedding theorem (see Proposition \ref{sobolev critical all dimension} and Corollary \ref{optimal embedding theorem}). These results enhance earlier results due to  Hansson \cite{Hansson}, Brezis and Wainger \cite{Brezis-Wainger},  and Cwikel  and Pustylnik \cite{Cwikel-Pustylnik}.

Thirdly, we show that Solomyak's version of Cwikel-Solomyak estimate given in \cite{Solomyak1995} is optimal in the class of Orlicz spaces (see Corollary \ref{optimality within orlicz}).

We explain the details of our contribution in Section \ref{known section} below after explaining the notations.

\section{Preliminaries}

\subsection{Symmetric function spaces}

For detailed information about decreasing rearrangements and symmetric spaces  briefly discussed below,  we refer the reader to the standard textbook \cite{KPS} (see also  \cite{LT2, LSZ-book}). Let $(\Omega,\nu)$ be a measure space. Let $S(\Omega,\nu)$ be the collection of all $\nu$-measurable functions on $\Omega$ such that, for some $n\in\mathbb{N},$ the function $|f|\chi_{\{|f|>n\}}$ is supported on a set of finite measure. For every $f\in S(\Omega,\nu)$ one can define the notion of decreasing rearrangement of $f$ (denoted by $\mu(f)$). This is a positive decreasing function on $\mathbb{R}_+$ equimeasurable with $|f|.$ 

Let $E(\Omega,\nu)\subset S(\Omega,\nu)$ and let $\|\cdot\|_E$ be Banach norm on $E(\Omega,\nu)$ such that
\begin{enumerate}[{\rm (1)}]
\item if $f\in E(\Omega,\nu)$ and $g\in S(\Omega,\nu)$ be such $|g|\leq|f|,$ then $g\in E(\Omega,\nu)$ and $\|g\|_E\leq\|f\|_E;$
\item if $f\in E(\Omega,\nu)$ and $g\in S(\Omega,\nu)$ be such $\mu(g)=\mu(f),$ then $g\in E(\Omega,\nu)$ and $\|g\|_E=\|f\|_E;$
\end{enumerate}
We say that $(E(\Omega,\nu),\|\cdot\|_E)$ (or simply $E$) is a symmetric Banach function space (symmetric space, for brevity). The classical example of such spaces is given by Lebesgue $L_p$-spaces, 
($L_p(\Omega,\nu), \|\cdot\|_p)$,  $1\leq p\leq \infty.$

If $\Omega=\mathbb{R}_+,$ then the function
$$t\to\|\chi_{(0,t)}\|_E,\quad t>0,$$
is called the {\it fundamental} function of $E.$ Similar definition is available when $\Omega$ is an interval or an arbitrary $\sigma$-finite measure space.  The concrete examples of measure spaces $(\Omega, \nu)$ considered in this paper are $d$-dimensional torus $\mathbb{T}^d$ (equipped with Haar measure), $\mathbb{R}_+,$ $\mathbb{R}^d$ (equipped with Lebesgue measure), their measurable subsets and compact $d$-dimensional Riemannian manifolds $(X,g).$

Among concrete symmetric spaces used in this paper are $L_p$-spaces, Orlicz, Lorentz spaces and Marcinkiewicz spaces. Given a convex function $M$ on $\mathbb{R}_+$ (continuous at $0$) such that $M(0)=0,$ Orlicz space $L_M(\Omega,\nu)$ is defined by setting
$$L_M(\Omega,\nu)=\Big\{f\in S(\Omega,\nu):\ M(\frac{|f|}{\lambda})\in L_1(\Omega,\nu)\mbox{ for some }\lambda>0\Big\}.$$
We equip it with a Banach norm
$$\|f\|_{L_M}=\inf\Big\{\lambda>0:\ \Big\|M(\frac{|f|}{\lambda})\Big\|_1\leq 1\Big\}.$$
We refer the reader to \cite{KrasRut, KPS, LT2} for further information about Orlicz spaces.

Given a concave increasing function $\psi,$ Lorentz space $\Lambda_{\psi}$ is defined by setting
$$\Lambda_{\psi}(\Omega,\nu)=\Big\{f\in S(\Omega,\nu):\ \int_0^{\infty}\mu(s,f)d\psi(s)<\infty\Big\}.$$
Marcinkiewicz space $\mathtt{M}_{\psi}$ is defined by setting
$$\mathtt{M}_{\psi}(\Omega,\nu)=\Big\{f\in S(\Omega,\nu):\ \sup_{t>0}\frac1{\psi(t)}\int_0^t\mu(s,f)ds<\infty\Big\}.$$
Spaces $L_{2,1}$ and $L_{2,\infty}$ are Lorentz and Marcinkiewicz space with $\psi(t)=t^{\frac12},$ $t>0$ respectively (see e.g. \cite{LT2}).

{\color{blue} All mentioned examples are closed with respect to the Hardy-Littlewood submajorization. The latter is a pre-order defined on $L_1+L_{\infty}$ by setting
$$\int_0^t\mu(s,y)ds\leq\int_0^t\mu(s,x)ds,\quad t>0,$$
and is denoted by $y\prec\prec x.$}

Let us fix throughout this paper  concave functions
$$\psi(t)=\frac1{\log(\frac{e}{t})},\quad {\rm  and}\quad \phi(t)=t\log(\frac{e}{t}), \quad t\in(0,1)$$  and consider Orlicz functions $$M(t)=t\log(e+t),\quad{\rm  and}\quad  G(t)=e^t-1,\quad t>0.$$
Throughout this text, the interplay between Lorentz spaces  $\Lambda_{\psi}(0,1)$ and $\Lambda_{\phi}(0,1)$, Orlicz spaces $L_M(0,1)$ and $L_G(0,1)$,  and Marcinkiewicz spaces  $\mathtt{M}_{\psi}(0,1)$ and $\mathtt{M}_{\phi}(0,1)$ play a fundamental role. The space $\mathtt{M}_{\psi}$ (respectively, the space $\Lambda_{\phi}$) is the largest (respectively, the smallest) symmetric Banach function space with the fundamental function $\phi.$

First of all, we observe that the spaces $\mathtt{M}_{\phi}(0,1)$ and $L_G(0,1)$ coincide (the corresponding norms are equivalent). Indeed, the fundamental functions of those spaces coincide, hence $L_G(0,1)\subset\mathtt{M}_{\phi}(0,1).$ On the other hand, $\phi'\in L_G(0,1)$ and, therefore, $\mathtt{M}_{\phi}(0,1)\subset L_G(0,1).$ This fact was first observed in \cite{AstSukIsrJ}.

By invoking K\"othe duality (see \cite{KPS} and \cite{LT2}), we infer that the spaces $\Lambda_\phi(0,1)$ and $L_M(0,1)$ also coincide and (the corresponding norms are equivalent).

In the statement of our main result, we use the notion of $2$-convexification of the symmetric space. Given a symmetric space $E(\Omega,\nu),$ we set
$$E^{(2)}(\Omega,\nu)=\{f\in S(\Omega,\nu):\ |f|^2\in E(\Omega,\nu)\},$$
$$\|f\|_{E^{(2)}}=\big\||f|^2\big\|_E^{\frac12},\quad f\in E^{(2)}(\Omega,\nu).$$
More details on the $2$-convexification of a function space may be found in the book  \cite{LT2}.

We also need a definition of dilation operator $\sigma_u,$ $u>0,$ which acts on $S(\mathbb{R})$ by the formula
$$(\sigma_uf)(t)=f(\frac{t}{u}),\quad f\in S(\mathbb{R}).$$

We frequently use the mapping $r_d:\mathbb{R}^d\to\mathbb{R}_+$ defined by the formula
$$r_d(t)=|t|^d,\quad t\in\mathbb{R}^d.$$

\section{Connection with earlier results}\label{known section}

\subsection{Connection with Cwikel-type estimates}

Let $d\in\mathbb{N}$ and let $f$ be a measurable function on $d$-dimensional torus $\mathbb{T}^d$ and let $M_f$ be an (unbounded) multiplication operator by $f$ on $L_2(\mathbb{T}^d).$ Let $\Delta_{\mathbb{T}^d}$ be the torical Laplacian. In Proposition \ref{torus boundedness} we show that (symmetrized) Cwikel operator
$$(1-\Delta_{\mathbb{T}^d})^{-\frac{d}{4}}M_f(1-\Delta_{\mathbb{T}^d})^{-\frac{d}{4}}$$ 
is bounded for $f\in\mathtt{M}_{\psi}(\mathbb{T}^d).$

Symmetrized Cwikel estimate in the weak-trace ideal $\mathcal{L}_{1,\infty}$ as well as in the ideal $\mathcal{M}_{1,\infty}$ are important in Non-commutative Geometry and in Mathematical Physics. For the background concerning weak trace ideals $\mathcal{L}_{p,\infty}$, $1\leq p<\infty$ and Marcinkie\-wicz ideal $\mathcal{M}_{1,\infty}$ we refer to \cite{LSZ-book} (see also subsection \ref{optimality-subsection} below).

Symmetrized Cwikel estimate in $\mathcal{L}_{1,\infty}$ (on the torus) for {\it even} $d$ appeared in the foundational papers \cite{Solomyak1994,Solomyak1995}. The estimate there was given in terms of the Orlicz norm $\|\cdot\|_{L_M}$ (which is equivalent to $\|\cdot\|_{\Lambda_{\phi}}$). Recently, symmetrized Cwikel estimates in the ideal $\mathcal{M}_{1,\infty}$ were established (on the Euclidean space) in \cite{LSZ-last-kalton}. The estimate was given in terms of the Lorentz norm
$$f\to\|\mu(f)\chi_{(0,1)}\|_{\Lambda_{\phi}(0,1)}+\|f\|_{L_1(\mathbb{R}^d)}.$$
The surprising fact that $\Lambda_\phi(0,1)=L_M(0,1)$ demonstrates the convergence of those totally unrelated approaches.

The results in this paper complement the results cited in the preceding paragraph. Indeed, Cwikel operator belongs to $\mathcal{L}_{\infty}$ for $f\in\mathtt{M}_{\psi}$ and to $\mathcal{L}_{1,\infty}$ for $f\in\Lambda_{\phi}.$  This opens an avenue (using interpolation methods) to determine the least ideal to which the Cwikel operator belongs.

\subsection{Connection with Sobolev-type estimates}

The following Sobolev inequality is customarily credited to \cite{Hansson} and \cite{Brezis-Wainger}. Actually, none of those papers contain a complete proof or even a clear-cut statement. The proof is available in \cite{Cwikel-Pustylnik}. For a notion of Sobolev space on $\Omega\subset\mathbb{R}^d$ we refer the reader to Chapter 7 in \cite{Adams}.

\begin{thm}\label{thm 1} Let $d$ be even and let $\Omega$ be a bounded domain in $\mathbb{R}^d$  (conditions apply). There exists a constant $c_{\Omega}$ such that
$$\|u\|_{\Lambda_{\psi}^{(2)}}\leq c_{\Omega}\|u\|_{W^{\frac{d}{2},2}},\quad u\in W^{\frac{d}{2},2}(\Omega).$$
\end{thm}

This result is optimal due to  the following theorem (proved in \cite{Cwikel-Pustylnik}).

\begin{thm}\label{thm 2} Let $d$ be even and let $\Omega$ be a bounded domain in $\mathbb{R}^d$  (conditions apply). If $X$ be a Banach symmetric function space on $\Omega$ such that
$$\|u\|_X\leq c_{\Omega}\|u\|_{W^{\frac{d}{2},2}},\quad u\in W^{\frac{d}{2},2}(\Omega)$$
for some constant $c_{\Omega}$, 
then $\Lambda_{\psi}^{(2)}\subset X.$
\end{thm}

In the next sections, we prove the assertions which substantially strengthen theorems above and extend them to arbitrary dimensions $d\ge 1$. In the course of the proof, we also recast Sobolev inequality using distribution functions. Those results are very much inspired by \cite{Cwikel-Pustylnik} {\color{blue} and our technique is a substantial improvement of that in \cite{Cwikel-Pustylnik}.}

\section{Distributional Sobolev-type inequality}

In this section, we introduce Hardy-type operator $T$ (see e.g. \cite{KPS}) and employ it as a technical tool for our distributional version of Sobolev inequality. In the second subsection, we show that our result is optimal.

\subsection{Distributional Sobolev-type inequality}

Let $T:L_2(0,1)\to L_2(0,1)$ be the operator defined by the formula
$$(Tx)(t)=t^{-\frac12}\int_0^tx(s)ds+\int_t^1\frac{x(s)ds}{s^{\frac12}},\quad x\in L_2(0,1).$$
The boundedness of $T:L_2(0,1)\to L_2(0,1)$ is guaranteed by the fact that  $T$ is actually a Hilbert-Schmidt operator. Indeed, its integral kernel is given by the formula
$$(t,s)\to t^{-\frac12}\chi_{\{s<t\}}+s^{-\frac12}\chi_{\{t<s\}},\quad 0<s,t<1,$$
which is, obviously, square-integrable.

The gist of (the critical case of the) Sobolev inequality on $\mathbb{R}^d$  may be informally understood as the boundedness of the operator $(1-\Delta)^{-\frac{d}{4}}$ from $L_2(\mathbb{R}^d)$ into a symmetric Banach function space on $\mathbb{R}^d$ which is  sufficiently close to $L_{\infty}(\mathbb{R}^d).$ We suggest to view this result as a statement concerning distribution functions of elements $(1-\Delta)^{-\frac{d}{4}}x$ and $T(\mu(x)\chi_{(0,1)})$ where $x\in L_2(\mathbb{R}^d).$

Our distributional version of Sobolev inequality is as follows:

\begin{thm}\label{dist sobolev critical} Let $d\in\mathbb{N}.$ For every $x\in L_2(\mathbb{R}^d),$ we have
$$\mu\big((1-\Delta)^{-\frac{d}{4}}x\big)\chi_{(0,1)}\prec\prec c_dT(\mu(x)\chi_{(0,1)}).$$
Here, $\prec\prec$ denotes the Hardy-Littlewood submajorization.
\end{thm}

We need the following lemma stated in terms of convolutions.

\begin{lem}\label{oneil lemma} Let $x\in L_2(\mathbb{R}^d)$ and $g\in (L_{2,\infty}\cap L_1)(\mathbb{R}^d).$ We have
$$\mu(x\ast g)\chi_{(0,1)}\prec\prec 4\|g\|_{L_{2,\infty}\cap L_1}T(\mu(x)\chi_{(0,1)}).$$
\end{lem}
\begin{proof} Lemma 1.5 in \cite{oneil} states that
$$\int_0^t\mu(s,x\ast g)ds\leq \int_0^t\mu(s,x)ds\cdot \int_0^t\mu(s,g)ds+t\int_t^{\infty}\mu(s,x)\mu(s,g)ds.$$
Obviously,
$$\int_0^t\mu(s,x)ds\leq \|g\|_{2,\infty}\cdot \int_0^ts^{-\frac12}ds=2t^{\frac12}\|g\|_{2,\infty}\leq 2\|g\|_{L_{2,\infty}\cap L_1}\cdot t^{\frac12}.$$
For $t\in(0,1),$ we have
$$t\int_t^{\infty}\mu(s,x)\mu(s,g)ds=t\int_t^1\mu(s,x)\mu(s,g)ds+t\int_1^{\infty}\mu(s,x)\mu(s,g)ds\leq$$
$$\leq t\|g\|_{2,\infty}\int_t^1\mu(s,x)\frac{ds}{s^{\frac12}}+t\mu(1,x)\|g\|_1\leq$$
$$\leq \|g\|_{L_{2,\infty}\cap L_1}\cdot \Big(t\int_t^1\mu(s,x)\frac{ds}{s^{\frac12}}+t\mu(1,x)\Big).$$
We, therefore, have
\begin{equation}\label{oneil eq0}
\int_0^t\mu(s,x\ast g)ds\leq 2\|g\|_{L_{2,\infty}\cap L_1}\cdot F(t),\quad t\in(0,1),
\end{equation}
where
$$F(t)=t^{\frac12}\int_0^t\mu(s,x)ds+t\int_t^1\mu(s,x)\frac{ds}{s^{\frac12}}+t\mu(1,x),\quad t>0.$$
Computing the derivative, we obtain
$$F'(t)=\frac12t^{-\frac12}\int_0^t\mu(s,x)ds+\int_t^1\mu(s,x)\frac{ds}{s^{\frac12}}+\mu(1,x)\leq $$
$$\leq (T\mu(x)\chi_{(0,1)})(t)+\mu(1,x).$$
Next,
$$(T\mu(x)\chi_{(0,1)})(t)\geq t^{-\frac12}\int_0^t\mu(1,x)ds+\int_t^1\frac{\mu(1,x)ds}{s^{\frac12}}=$$
$$=\mu(1,x)\cdot (2-t^{\frac12})\geq\mu(1,x),\quad t\in(0,1).$$
Thus,
$$F'(t)\leq 2(T\mu(x)\chi_{(0,1)})(t),\quad t\in(0,1).$$
It is now immediate that
$$F(t)\leq 2\int_0^t\Big((T\mu(x)\chi_{(0,1)})(s)\Big)ds,\quad t\in(0,1).$$
Combining this equation with \eqref{oneil eq0}, we obtain
$$\int_0^t\mu(s,x\ast g)ds\leq 4\|g\|_{L_{2,\infty}\cap L_1}\int_0^t\Big((T\mu(x)\chi_{(0,1)})(s)\Big)ds,\quad t\in(0,1).$$
This is exactly the required assertion.
\end{proof}

\begin{proof}[Proof of Theorem \ref{dist sobolev critical}] We rewrite $(1-\Delta)^{-\frac{d}{4}}$ as a convolution operator. Namely,
$$(1-\Delta)^{-\frac{d}{4}}x=x\ast g,\quad x\in L_2(\mathbb{R}^d),$$
where $g$ is the Fourier transform of the function
$$t\to (1+|t|^2)^{-\frac{d}{4}},\quad t\in\mathbb{R}^d.$$
Precise expression for the function $g$ involves Macdonald function $K_{\frac{d}{4}}$ and is given in \cite{AronszajnSmith} (see formulae (2.7) and (2.10) there) as follows
$$g(t)=c_d|t|^{-\frac{d}{4}}K_{\frac{d}{4}}(|t|),\quad t\in\mathbb{R}^d.$$

Let $\mathbb{B}^d$ be the unit ball in $\mathbb{R}^d.$ If $d\neq0{\rm mod}4,$ it follows from formulae (9.6.2) and (9.6.10) in \cite{AbramovitzStegun} that $g\chi_{\mathbb{B}^d}\in L_{2,\infty}(\mathbb{B}^d)\subset L_1(\mathbb{B}^d).$ If $d=0{\rm mod}4,$  formula (9.6.11) in \cite{AbramovitzStegun} yields that $g\chi_{\mathbb{B}^d}\in L_{2,\infty}(\mathbb{B}^d)\subset L_1(\mathbb{B}^d).$ Also, $g\chi_{\mathbb{R}^d\backslash\mathbb{B}^d}\in (L_1\cap L_{\infty})(\mathbb{R}^d\backslash\mathbb{B}^d)$ by formula (9.7.2) in \cite{AbramovitzStegun}. Thus, $g\in (L_{2,\infty}\cap L_1)(\mathbb{R}^d).$ The assertion follows from Lemma \ref{oneil lemma}.
\end{proof}

\subsection{Optimality of the distributional Sobolev-type inequality }

We start with the following easy observation.

\begin{lem}\label{spherical rearrangement lemma} For $x\in S(0,\infty),$ we have
$$\mu(x\circ r_d)=\sigma_{\omega_d}\mu(x),\quad r_d(t)=|t|^d,\quad t\in\mathbb{R}^d,$$
where $\omega_d$ is the volume of the unit ball in $\mathbb{R}^d.$
\end{lem}
\begin{proof}  Indeed, for every interval $(a,b),$ we have
$$m(r_d^{-1}(a,b))=m(\{t\in\mathbb{R}^d:\ a<|t|^d<b\})=$$
$$=m(\{t\in\mathbb{R}^d:\ |t|<b^{\frac1d}\})-m(\{t\in\mathbb{R}^d:\ |t|<a^{\frac1d}\})=\omega_d\cdot (b-a).$$
Hence, for every set $A\subset(0,\infty),$ we have
$$m(r_d^{-1}(A))=\omega_dm(A).$$
In other words, the mapping $r_d:\mathbb{R}^d\to\mathbb{R}$ preserves a measure (modulo a constant factor $\omega_d$). This suffices to conclude the argument.
\end{proof}

The following proposition shows (with the help of Lemma \ref{spherical rearrangement lemma}) that Theorem \ref{dist sobolev critical} is optimal.
\begin{prop}\label{estimate from below lemma}  There exists a strictly positive constant $c_d'$ depending only on $d$ such that 
$$(1-\Delta)^{-\frac{d}{4}}(x\circ r_d)\geq c_d'(Tx)\circ r_d$$
for every $0\leq x\in L_2(0,\infty)$ supported in the interval $(0,1).$
\end{prop}
\begin{proof} As in the proof of Theorem \ref{dist sobolev critical} above, we rewrite $(1-\Delta)^{-\frac{d}{4}}$ as a convolution operator. Namely,
$$(1-\Delta)^{-\frac{d}{4}}z=z\ast g,\quad z\in L_2(\mathbb{R}^d)$$
where $g$ is given by the formula
$$g(t)=c_d|t|^{-\frac{d}{4}}K_{\frac{d}{4}}(|t|),\quad t\in\mathbb{R}^d.$$
By formula (9.6.23) in \cite{AbramovitzStegun}, this is a strictly positive function.	
	
Let $\mathbb{B}^d$ be the unit ball in $\mathbb{R}^d.$ Since $g$ is strictly positive, it follows from the formulae (9.6.2) and (9.6.10) in \cite{AbramovitzStegun} (when $d\neq0{\rm mod}4$) or from the formula (9.6.11) in \cite{AbramovitzStegun} (when $d=0{\rm mod}4$) that
$$g(t)\geq c_d'|t|^{-\frac{d}{2}},\quad 0<|t|<2.$$
	
Since $x\geq0,$ it follows that
$$\big((1-\Delta)^{-\frac{d}{4}}(x\circ r_d)\big)(t)=((x\circ r_d)\ast g)(t)\geq c_d'\int_{|t-s|<2}|t-s|^{-\frac{d}{2}}x(|s|^d)ds.$$
When $|t|<1,$ we have
$$\Big\{s\in\mathbb{R}^d:\quad |s|<1\Big\}\subset \Big\{s\in\mathbb{R}^d: |t-s|<2\Big\}.$$
Thus,
$$\big((1-\Delta)^{-\frac{d}{4}}(x\circ r_d)\big)(t)\geq c_d'\int_{|s|<1}|t-s|^{-\frac{d}{2}}x(|s|^d)ds,\quad |t|<1.$$
Obviously,
$$|t-s|\leq|t|+|s|\leq2\max\{|t|,|s|\}$$
and, therefore,
$$|t-s|^{-\frac{d}{2}}\geq 2^{-\frac{d}{2}}\min\{|t|^{-\frac{d}{2}},|s|^{-\frac{d}{2}}\}.$$
It follows that
$$\big((1-\Delta)^{-\frac{d}{4}}(x\circ r_d)\big)(t)\geq 2^{-\frac{d}{2}}c_d'\int_{|s|<1}\min\{|t|^{-\frac{d}{2}},|s|^{-\frac{d}{2}}\}x(|s|^d)ds,\quad |t|<1.$$

Passing to spherical coordinates, we obtain
$$\int_{|s|<1}\min\{|t|^{-\frac{d}{2}},|s|^{-\frac{d}{2}}\}x(|s|^d)ds=$$
$$=\int_0^1\min\{|t|^{-\frac{d}{2}},r^{-\frac{d}{2}}\}r^{d-1}x(r^d)dr\cdot\int_{\mathbb{S}^{d-1}}ds.$$
Making the substitution $r^d=u,$ we write
$$\int_0^1\min\{|t|^{-\frac{d}{2}},r^{-\frac{d}{2}}\}r^{d-1}x(r^d)dr=\frac1d\int_0^1\min\{|t|^{-\frac{d}{2}},u^{-\frac12}\}x(u)du=$$
$$=\frac1d\Big(\int_0^{|t|^d}|t|^{-\frac{d}{2}}x(u)du+\int_{|t|^d}^1u^{-\frac12}x(u)du\Big)=\frac1d(Tx)(|t|^d).$$
Combining the last $2$ equations, we complete the proof.
\end{proof}

\section{Sobolev embedding theorem in arbitrary dimension}

In this section, we extend Theorems \ref{thm 1} (see Proposition \ref{sobolev critical all dimension}) and \ref{thm 2} (see Corollary \ref{optimal embedding theorem}) to Euclidean spaces of arbitrary dimension. We provide new proofs of both theorems in their original setting (for both even and odd $d$).

In what follows, we extend $\psi$ to a concave increasing function on $(0,\infty)$ by setting
$$
\psi(t)=
\begin{cases}
\frac1{\log(\frac{e}{t})},& t\in(0,1)\\
t,& t\in[1,\infty).
\end{cases}
$$

\begin{prop}\label{sobolev critical all dimension} Let $d\in\mathbb{N}.$ For every $x\in L_2(\mathbb{R}^d),$ we have
$$\|(1-\Delta)^{-\frac{d}{4}}x\|_{\Lambda_{\psi}^{(2)}}\leq c_d\|x\|_2.$$
\end{prop}

With some effort, the crucial lemma below can be inferred from the proof of Theorem 5.1 in \cite{Cwikel-Pustylnik-JFA}. We provide a direct proof. For the definition of real interpolation method employed below we refer the reader to \cite{BerghLoefstrom}.
\begin{lem}\label{real interpolation lemma} We have
$$[\mathtt{M}_{\phi},L_{\infty}]_{\frac12,2}\subset\Lambda_{\psi}^{(2)}.$$
\end{lem}
\begin{proof} Let $t\in(0,1)$ and set $s=\psi^{-1}(t).$
	
Let $x\in\mathtt{M}_{\phi}+L_{\infty}$ and let $x=f+g,$ where $f\in\mathtt{M}_{\phi}$ and where $g\in L_{\infty}.$ Set
$$A=|x|^{-1}([\mu(s,x),\infty)),\quad B=|f|^{-1}([\frac12\mu(s,x),\infty)),\quad C=|g|^{-1}([\frac12\mu(s,x),\infty)).$$
Note that $m(A)\geq s.$

We have
$$|f|(u)<\frac12\mu(s,x),\quad |g|(u)<\frac12\mu(s,x),\quad u\in B^c\cap C^c.$$
Thus,
$$|x|(u)\leq |f|(u)+|g|(u)<\mu(s,x),\quad u\in B^c\cap C^c.$$
So, $u\in B^c\cap C^c$ implies that $u\in A^c.$ That is, $B^c\cap C^c\subset A^c$ or, equivalently, $A\subset B\cup C.$
	
If $m(B)\geq\frac12 m(A),$ then
$$\|f\|_{M_{\phi}}+t\|g\|_{L_{\infty}} \geq \|f\chi_B\|_{M_{\phi}} \geq \frac12\mu(s,x) \|\chi_B\|_{M_{\phi}} =$$
$$= \frac12\mu(s,x) \psi(m(B)) \geq \frac14\mu(s,x)\phi(m(A)) \geq \frac14 t\mu(s,x).$$
If $m(C)\geq \frac12 m(A),$ then
$$\|f\|_{M_{\phi}}+t\|g\|_{L_{\infty}} \geq t\|g\chi_C\|_{L_{\infty}} \geq \frac12t\mu(s,x).$$
In either case, we have
$$K(t,x,\mathtt{M}_{\phi},L_{\infty})\geq \frac14t\mu(s,x).$$
	
Thus,
$$\|x\|_{[\mathtt{M}_{\phi},L_{\infty}]_{\frac12,2}}^2=\int_0^1 (\frac1tK(t,x,\mathtt{M}_{\phi},L_{\infty}))^2 dt \geq$$
$$\geq\frac1{16}\int_0^1 \mu^2(s,x)dt \stackrel{t=\psi(s)}{=}\frac1{16} \int_0^1\mu^2(s,x)d\psi(s).$$
\end{proof}

The following lemma shows that the receptacle of the operator $T$ is strictly smaller than $\exp(L_2)$ (the space suggested by the Moser-Trudinger inequality).

\begin{lem}\label{t lemma} We have $T:L_2(0,1)\to\Lambda_{\psi}^{(2)}(0,1).$ 
\end{lem}
\begin{proof} Let $x\in L_{2,1}(0,1).$ It is immediate that
$$|(Tx)(t)|\leq t^{-\frac12}\int_0^t|x(s)|ds+\int_t^1|x(s)|\frac{ds}{s^{\frac12}}\leq$$
$$\leq t^{-\frac12}\int_0^t\mu(s,x)ds+\int_0^1|x(s)|\frac{ds}{s^{\frac12}}\leq t^{-\frac12}\int_0^t\mu(s,x)ds+\int_0^1\mu(s,x)\frac{ds}{s^{\frac12}}=$$
$$=t^{-\frac12}\int_0^t\mu(s,x)ds+2\|x\|_{2,1}\leq t^{-\frac12}\|x\|_{2,\infty}\int_0^t\frac{ds}{s^{\frac12}}+2\|x\|_{2,1}=$$
$$=2\|x\|_{2,\infty}+2\|x\|_{2,1}\leq c_{{\rm abs}}\|x\|_{2,1}.$$
Thus,
$$\|T\|_{L_{2,1}\to L_{\infty}}\leq c_{{\rm abs}}.$$

Let $x\in L_{2,\infty}(0,1).$ It is immediate that
$$|(Tx)(t)|\leq\int_0^1|x(s)|\cdot\min\{t^{-\frac12},s^{-\frac12}\}ds\leq$$
$$\leq\int_0^1\mu(s,x)\cdot\min\{t^{-\frac12},s^{-\frac12}\}ds\leq\|x\|_{2,\infty}\int_0^1s^{-\frac12}\cdot\min\{t^{-\frac12},s^{-\frac12}\}ds.$$
Obviously,
$$\int_0^1s^{-\frac12}\cdot\min\{t^{-\frac12},s^{-\frac12}\}ds=t^{-\frac12}\int_0^t\frac{ds}{s^{\frac12}}+\int_t^1\frac{ds}{s}=\log(\frac{e^2}{t}).$$
That is,
$$|(Tx)(t)|\leq\|x\|_{2,\infty}\log(\frac{e^2}{t}),\quad t\in(0,1).$$
Since the mapping $t\to\log(\frac{e^2}{t}),$ $t\in(0,1),$ falls into $\exp(L_1)(0,1),$ it follows that
$$\|T\|_{L_{2,\infty}\to \exp(L_1)}\leq c_{{\rm abs}}.$$

By real interpolation, we have
$$T:[L_{2,\infty},L_{2,1}]_{\frac12,2}\to [\exp(L_1),L_{\infty}]_{\frac12,2}$$
is a bounded mapping. By Lemma \ref{real interpolation lemma}, we have
$$[L_{2,\infty},L_{2,1}]_{\frac12,2}=L_2,\quad [\exp(L_1),L_{\infty}]_{\frac12,2}\stackrel{L.\ref{real interpolation lemma}}{\subset}\Lambda_{\psi}^{(2)}.$$
Thus, $T:L_2(0,1)\to\Lambda_{\psi}^{(2)}(0,1)$ is a bounded mapping.
\end{proof}

\begin{proof}[Proof of Proposition \ref{sobolev critical all dimension}] By Theorem \ref{dist sobolev critical} and Lemma \ref{t lemma}, we have
$$\Big\|\mu\big((1-\Delta)^{-\frac{d}{4}}x\big)\chi_{(0,1)}\Big\|_{\Lambda_{\psi}^{(2)}}\leq c_d\|x\|_2.$$
In other words,
$$\Big\|(1-\Delta)^{-\frac{d}{4}}x\Big\|_{\Lambda_{\psi}^{(2)}+L_{\infty}}\leq c_d\|x\|_2.$$
On the other hand, we have
$$\Big\|(1-\Delta)^{-\frac{d}{4}}x\Big\|_2\leq \|x\|_2.$$
Thus,
$$\Big\|(1-\Delta)^{-\frac{d}{4}}x\Big\|_{(\Lambda_{\psi}^{(2)}+L_{\infty})\cap L_2}\leq c_d\|x\|_2.$$
Obviously,
$$(\Lambda_{\psi}^{(2)}+L_{\infty})\cap L_2=\Lambda_{\psi}^{(2)}$$
and the assertion follows.
\end{proof}

The next assertion should be compared with Theorem 4 in \cite{Cwikel-Pustylnik} (it is proved in the companion paper \cite{Cwikel-Pustylnik-JFA} and constitutes the key part of the proof of Theorem 5.7 in that paper). Note that our $T$ is different from that in \cite{Cwikel-Pustylnik}. This  difference is the reason why our proof is so much simpler.

\begin{thm}\label{t theorem} For every $z=\mu(z)\in \Lambda_{\psi}^{(2)}(0,1),$ there exists $x=\mu(x)\in L_2(0,1)$ such that
$$\mu(z)\leq Tx\mbox{ and }\|x\|_2\leq c_{{\rm abs}}\|z\|_{\Lambda_{\psi}^{(2)}}.$$
\end{thm}

The technical part of the proof of Theorem \ref{t theorem} is concentrated in the next lemma.
\begin{lem}\label{cwikel-pustylnik lorentz lemma} Let $z\in\Lambda_{\psi}^{(2)}(0,1)$ and let
$$y(t)=\sup_{0<s<t}\psi(s)\mu(s,z),\quad 0<t<1.$$
We have
$$\|y\|_{L_2((0,1),\frac{dt}{t})}\leq 3^{\frac12}\|z\|_{\Lambda_{\psi}^{(2)}}.$$
\end{lem}
\begin{proof} We claim that (here $C$ is the classical Cesaro operator)
$$\int_0^1\sup_{0<s<t}\psi^2(s)(C\mu(w))(s)\frac{dt}{t}\leq 3\|w\|_{\Lambda_{\psi}},\quad w\in\Lambda_{\psi}.$$
The functional on the left hand side is normal, subadditive and positively homogeneous. By Lemma II.5.2 in \cite{KPS}, it suffices to prove the inequality for the indicator functions.

Let $w=\chi_A$ and let $m(A)=u.$ Obviously, $\mu(w)=\chi_{(0,u)}$ and
$$(C\mu(w))(s)=\frac{\min\{u,s\}}{s},\quad 0<s<1.$$
Thus,
$$\int_0^1\sup_{0<s<t}\psi^2(s)(C\mu(w))(s)\frac{dt}{t}=\int_0^u\sup_{0<s<t}\psi^2(s)\frac{dt}{t}+$$
$$+\int_u^1\max\Big\{\sup_{0<s<u}\psi^2(s),\sup_{u<s<t}\psi^2(s)\frac{u}{s}\Big\}\frac{dt}{t}=$$
$$=\int_0^u\psi^2(t)\frac{dt}{t}+\int_u^1\max\Big\{\psi^2(u),u\sup_{u<s<t}s^{-1}\psi^2(s)\Big\}\frac{dt}{t}.$$
The function $s\to s^{-1}\psi^2(s),$ $s\in(0,1),$ decreases on the interval $(0,e^{-1})$ and increases on the interval $(e^{-1},1).$ Thus,
$$\sup_{u<s<t}s^{-1}\psi^2(s)=
\begin{cases}
u^{-1}\psi^2(u),& t\in(0,e^{-1})\\
t^{-1}\psi^2(t),& u\in(e^{-1},1)\\
\frac{e}{4},& u<e^{-1}<t
\end{cases}
$$
For $u\in(e^{-1},1),$ we have
$$\int_0^1\sup_{0<s<t}\psi^2(s)(C\mu(w))(s)\frac{dt}{t}=\int_0^u\psi^2(t)\frac{dt}{t}+\int_u^1\max\Big\{\psi^2(u),ut^{-1}\psi^2(t)\Big\}\frac{dt}{t}=$$
$$=\int_0^u\psi^2(t)\frac{dt}{t}+\int_u^1\psi^2(u)\frac{dt}{t}=\psi(u)+\psi^2(u)\log(\frac1u)\leq 2\psi(u).$$
For $u\in(0,e^{-1}),$ we have
$$\int_0^1\sup_{0<s<t}\psi^2(s)(C\mu(w))(s)\frac{dt}{t}=\int_0^u\psi^2(t)\frac{dt}{t}+\int_u^{e^{-1}}\psi^2(u)\frac{dt}{t}+$$
$$+\int_{e^{-1}}^1\max\Big\{\psi^2(u),\frac{e}{4}u\Big\}\frac{dt}{t}=\psi(u)+\psi^2(u)\log(\frac1{eu})+\max\Big\{\psi^2(u),\frac{e}{4}u\Big\}\leq 3\psi(u).$$
This yields the claim.

Let $w=\mu^2(z)\in\Lambda_{\psi}.$ We have
$$\int_0^1\sup_{0<s<t}\psi^2(s)\mu(s,w)\frac{dt}{t}\leq \int_0^1\sup_{0<s<t}\psi^2(s)(C\mu(w))(s)\frac{dt}{t}\leq 3\|w\|_{\Lambda_{\psi}}.$$
The assertion follows immediately.
\end{proof}

The following lemma affords a very substantial simplification and streamlining of the arguments employed in \cite{Cwikel-Pustylnik}.

\begin{lem}\label{l2 to l2 lemma} Let $y\in L_2((0,1),\frac{dt}{t})$ and let $x(t)=t^{-\frac12}y(t),$ $0<t<1.$ If $y$ is positive and increasing, then
$$\psi(t)\cdot (Tx)(t)\geq \frac1{4e}y(\frac{t}{2}),\quad 0<t<1.$$
\end{lem}
\begin{proof} Suppose $t\in(0,e^{-1}).$ We have
$$\psi(t)\cdot (Tx)(t)\geq\psi(t)\int_t^1\frac{y(s)ds}{s}\geq y(t)\cdot \psi(t)\int_t^1\frac{ds}{s}\geq\frac12 y(t)\geq \frac1{4e}y(\frac{t}{2}).$$
Suppose $t\in(e^{-1},1).$ We have
$$\psi(t)\cdot (Tx)(t)\geq \frac12(Tx)(t)\geq \frac1{2t^{\frac12}}\int_0^ty(s)ds\geq\frac12\int_{\frac{t}{2}}^ty(s)ds\geq \frac1{4e}y(\frac{t}{2}).$$
\end{proof}

\begin{proof}[Proof of Theorem \ref{t theorem}] Let $y$ be as in Lemma \ref{cwikel-pustylnik lorentz lemma} and let $x$ be as in Lemma \ref{l2 to l2 lemma}. It follows from Lemma \ref{cwikel-pustylnik lorentz lemma} that $\|x\|_2\leq 3^{\frac12}\|z\|_{\Lambda_{\psi}^{(2)}}.$ Obviously, $y$ is positive and increasing. It follows from Lemma \ref{l2 to l2 lemma} that
$$\psi(t)\cdot (Tx)(t)\geq \frac1{4e}y(\frac{t}{2})\geq \frac1{4e}\psi(\frac{t}{2})\mu(\frac{t}{2},z)\geq \frac1{8e}\psi(t)\mu(t,z),\quad 0<t<1.$$
Thus,
$$T\mu(x)\geq Tx\geq \frac1{8e}\mu(z).$$	
\end{proof}

The next corollary shows that the result of Proposition \ref{sobolev critical all dimension} is optimal.

\begin{cor}\label{optimal embedding theorem} Let $E(\mathbb{R}^d)$ be a symmetric Banach function space on $\mathbb{R}^d.$ If
$$(1-\Delta)^{-\frac{d}{4}}:L_2(\mathbb{R}^d)\to E(\mathbb{R}^d),$$
then $\Lambda_{\psi}^{(2)}(\mathbb{R}^d)\subset (E+L_{\infty})(\mathbb{R}^d).$
\end{cor}

\section{Degenerate case of Cwikel estimate}

In this case, we consider operators
$$(1-\Delta)^{-\frac{d}{4}}M_f(1-\Delta)^{-\frac{d}{4}}\mbox{ and }(1-\Delta_{\mathbb{T}^d})^{-\frac{d}{4}}M_f(1-\Delta_{\mathbb{T}^d})^{-\frac{d}{4}}$$
which act, respectively, on $L_2(\mathbb{R}^d)$ and $L_2(\mathbb{T}^d)$ and evaluate their uniform norms. We show that the maximal (symmetric Banach function) space $E$ such that the operators above are bounded for every $f\in E$ is the Marcinkiewicz space $\mathtt{M}_{\psi}.$ For Euclidean space, this follows from Propositions \ref{space boundedness} and \ref{space boundedness optimal} below. For torus, this follows from Propositions \ref{torus boundedness} and \ref{torus boundedness optimal} below.

\subsection{Estimates for Euclidean space}

\begin{prop}\label{space boundedness} Let $d\in\mathbb{N}.$ We have
$$\Big\|(1-\Delta)^{-\frac{d}{4}}M_f(1-\Delta)^{-\frac{d}{4}}\Big\|_{\infty}\leq c_d\|f\|_{\mathtt{M}_{\psi}},\quad f\in\mathtt{M}_{\psi}(\mathbb{R}^d).$$
\end{prop}
\begin{proof} Without loss of generality, $f$ is real-valued and positive. We have
$$\Big\|(1-\Delta)^{-\frac{d}{4}}M_f(1-\Delta)^{-\frac{d}{4}}\Big\|_{\infty}=\sup_{\|x\|_2\leq1}\Big|\langle (1-\Delta)^{-\frac{d}{4}}M_f(1-\Delta)^{-\frac{d}{4}}x,x\rangle\Big|=$$
$$=\sup_{\|x\|_2\leq1}\Big|\langle f\cdot (1-\Delta)^{-\frac{d}{4}}x,(1-\Delta)^{-\frac{d}{4}}x\rangle\Big|=$$
$$=\sup_{\|x\|_2\leq1}\Big\|f\cdot \Big|(1-\Delta)^{-\frac{d}{4}}x\Big|^2\Big\|_1\leq \sup_{\|x\|_2\leq1}\|f\|_{\mathtt{M}_{\psi}}\Big\|(1-\Delta)^{-\frac{d}{4}}x\Big\|_{\Lambda_{\psi}^{(2)}}^2.$$
The assertion follows now from Proposition \ref{sobolev critical all dimension}.
\end{proof}

The converse inequality follows from Proposition \ref{estimate from below lemma} and Theorem \ref{t theorem}.

\begin{prop}\label{space boundedness optimal} Let $d\in\mathbb{N}.$ Let $f=\mu(f)\in\mathtt{M}_{\psi}(0,\infty).$ We have
$$\Big\|(1-\Delta)^{-\frac{d}{4}}M_{f\circ r_d}(1-\Delta)^{-\frac{d}{4}}\Big\|_{\infty}\geq c_d\|f\|_{\mathtt{M}_{\psi}},.$$
\end{prop}
\begin{proof} We have
$$\Big\|(1-\Delta)^{-\frac{d}{4}}M_{f\circ r_d}(1-\Delta)^{-\frac{d}{4}}\Big\|_{\infty}=$$
$$=\sup_{\|\xi\|_2\leq1}\Big|\langle (1-\Delta)^{-\frac{d}{4}}M_{f\circ r_d}(1-\Delta)^{-\frac{d}{4}}\xi,\xi\rangle\Big|=$$
$$=\sup_{\|\xi\|_2\leq1}\Big|\langle (f\circ r_d)\cdot (1-\Delta)^{-\frac{d}{4}}\xi,(1-\Delta)^{-\frac{d}{4}}\xi\rangle\Big|=$$
$$=\sup_{\|\xi\|_2\leq1}\Big\|(f\circ r_d)\cdot \Big|(1-\Delta)^{-\frac{d}{4}}\xi\Big|^2\Big\|_1.$$

Let us now restrict the supremum to the radial $\xi.$ That is, let $\xi=x\circ r_d,$ where $x\in L_2(0,\infty)$ and $\|x\|_2\leq\omega_d^{-\frac12}.$ We have
$$\Big\|(1-\Delta)^{-\frac{d}{4}}M_{f\circ r_d}(1-\Delta)^{-\frac{d}{4}}\Big\|_{\infty}\geq$$
$$\geq\sup_{\|x\|_2\leq\omega_d^{-\frac12}}\Big\|(f\circ r_d)\cdot \Big|(1-\Delta)^{-\frac{d}{4}}(x\circ r_d)\Big|^2\Big\|_1.$$
Let us further assume that $x=\mu(x)$ is supported on the interval $(0,1).$ By Proposition \ref{estimate from below lemma} we have
$$\Big\|(1-\Delta)^{-\frac{d}{4}}M_{f\circ r_d}(1-\Delta)^{-\frac{d}{4}}\Big\|_{\infty}\geq$$
$$\geq\sup_{\substack{x=\mu(x)\\ \|x\|_2\leq\omega_d^{-\frac12}\\ x=0\mbox{ on }(1,\infty)}}\Big\|(f\circ r_d)\cdot \Big|(1-\Delta)^{-\frac{d}{4}}(x\circ r_d)\Big|^2\Big\|_1\geq$$
$$\geq c_d\cdot\sup_{\substack{x=\mu(x)\\ \|x\|_2\leq\omega_d^{-\frac12}\\ x=0\mbox{ on }(1,\infty)}}\Big\|(f\circ r_d)\cdot \Big|Tx\circ r_d\Big|^2\Big\|_1=c_d\omega_d\cdot\sup_{\substack{x=\mu(x)\\ \|x\|_2\leq\omega_d^{-\frac12}\\ x=0\mbox{ on }(1,\infty)}}\big\|f\cdot |Tx|^2\big\|_1.$$

By Theorem \ref{t theorem}, we have
$$\Big\|(1-\Delta)^{-\frac{d}{4}}M_{f\circ r_d}(1-\Delta)^{-\frac{d}{4}}\Big\|_{\infty}\geq$$
$$\geq c_d\omega_d\cdot\sup_{\substack{z=\mu(z)\in\Lambda_{\psi^{(2)}}(0,1)\\ \|z\|_{\Lambda_{\psi^{(2)}}}\leq c_{{\rm abs}}^{-1}\omega_d^{-\frac12}}}\big\|f\cdot z^2\big\|_1=c_dc_{{\rm abs}}^{-2}\cdot\sup_{\substack{z=\mu(z)\in\Lambda_{\psi}(0,1)\\ \|z\|_{\Lambda_{\psi}}\leq 1}}\big\|f\cdot z\big\|_1.$$

It is clear that
$$\sup_{\substack{z=\mu(z)\in\Lambda_{\psi}(0,1)\\ \|z\|_{\Lambda_{\psi}}\leq 1}}\big\|f\cdot z\big\|_1=\|f\chi_{(0,1)}\|_{\mathtt{M}_{\psi}}.$$
Since $f=\mu(f)$ and since $\psi$ is linear on $(1,\infty),$ it follows that
$$\|f\chi_{(0,1)}\|_{\mathtt{M}_{\psi}}=\|\mu(f)\chi_{(0,1)}\|_{\mathtt{M}_{\psi}}\approx_d\|f\|_{\mathtt{M}_{\psi}+L_{\infty}}\approx_d\|f\|_{\mathtt{M}_{\psi}}.$$
Combining three last equations, we complete the proof.
\end{proof}

\subsection{Estimates for the torus}

The following lemma is taken from \cite{DAO} (see Lemmas 4.5 and 4.6 there).

\begin{lem}\label{non-compact to compact lemma} Let $h$ be a measurable function on $[-1,1]^d.$ We have
$$M_h(1-\Delta)^{-\frac{d}{2}}M_h\Big|_{L_2([-1,1]^d)}=M_ha(\nabla_{\mathbb{T}^d})M_h\Big|_{L_2([-1,1]^d)},$$
where
$$a(n)=(1+|n|^2)^{-\frac{d}{2}}+b(n),\quad b(n)=O((1+|n|^2)^{-\frac{d+1}{2}}),\quad n\in\mathbb{Z}^d.$$
\end{lem}

The following lemma relies on the post-critical Sobolev inequality:
$$\Big\|(1-\Delta_{\mathbb{T}^d})^{-\frac{d+1}{4}}\Big\|_{L_2(\mathbb{T}^d)\to L_{\infty}(\mathbb{T}^d)}\leq c_d.$$
The validity of this inequality follows immediately from the fact
$$\Big\{(1+|n|^2)^{-\frac{d+1}{2}}\Big\}_{n\in\mathbb{Z}^d}\in l_1(\mathbb{Z}^d).$$

\begin{lem}\label{postcritical torus lemma} Let $h\in L_1(\mathbb{T}^d).$ We have
$$\Big\|(1-\Delta_{\mathbb{T}^d})^{-\frac{d+1}{4}}M_h(1-\Delta_{\mathbb{T}^d})^{-\frac{d+1}{4}}\Big\|_{\infty}\leq c_d\|h\|_1.$$
\end{lem}
\begin{proof} Without loss of generality, $h$ is real-valued and positive. We have
$$\Big\|(1-\Delta_{\mathbb{T}^d})^{-\frac{d+1}{4}}M_h(1-\Delta_{\mathbb{T}^d})^{-\frac{d+1}{4}}\Big\|_{\infty}=\sup_{\|x\|_2\leq1}\Big|\langle (1-\Delta_{\mathbb{T}^d})^{-\frac{d+1}{4}}M_h(1-\Delta_{\mathbb{T}^d})^{-\frac{d+1}{4}}x,x\rangle\Big|=$$
$$=\sup_{\|x\|_2\leq1}\Big|\langle h\cdot (1-\Delta_{\mathbb{T}^d})^{-\frac{d+1}{4}}x,(1-\Delta_{\mathbb{T}^d})^{-\frac{d+1}{4}}x\rangle\Big|=$$
$$=\sup_{\|x\|_2\leq1}\Big\|h\cdot \Big|(1-\Delta_{\mathbb{T}^d})^{-\frac{d+1}{4}}x\Big|^2\Big\|_1\leq \sup_{\|x\|_2\leq1}\|h\|_1\Big\|(1-\Delta_{\mathbb{T}^d})^{-\frac{d+1}{4}}x\Big\|_{\infty}^2.$$
The assertion follows now from the post-critical Sobolev inequality.
\end{proof}

It is of crucial importance that the estimate in the preceding lemma is given in terms of $\|h\|_1$ rather than $\|h\|_{\mathtt{M}_{\psi}}.$

The following proposition is a version of Proposition \ref{space boundedness} for $\mathbb{T}^d.$ Observe that Marcinkiewicz space $\mathtt{M}_{\psi}$ strictly contains the Orlicz space $L_M,$ $M(t)=t\log(e+t),$ $t>0. 	$ The result below is the best possible criterion for the boundedness of the symmetrized Cwikel operator 
$$(1-\Delta_{\mathbb{T}^d})^{-\frac{d}{4}}M_f(1-\Delta_{\mathbb{T}^d})^{-\frac{d}{4}}.$$
The sharpness of the result will be demonstrated below in Proposition \ref{torus boundedness optimal}.

\begin{prop}\label{torus boundedness} Let $d\in\mathbb{N}.$ Let $f\in\mathtt{M}_{\psi}(\mathbb{T}^d).$ We have
$$\Big\|(1-\Delta_{\mathbb{T}^d})^{-\frac{d}{4}}M_f(1-\Delta_{\mathbb{T}^d})^{-\frac{d}{4}}\Big\|_{\infty}\leq c_d\|f\|_{\mathtt{M}_{\psi}}.$$
\end{prop}
\begin{proof} Without loss of generality, $f$ is real-valued, positive and supported on $[-1,1]^d.$

In this proof, we frequently use the following property: let  $H_0\subset H$ be a Hilbert subspace and let $A:H\to H_0.$ If $A$ vanishes on the orthogonal complement of $H_0,$ then $\|A\|_{\infty}=\|A|_{H_0}\|_{\infty}.$

Firstly, note that
$$\Big\|(1-\Delta)^{-\frac{d}{4}}M_f(1-\Delta)^{-\frac{d}{4}}\Big\|_{\infty}=\Big\|M_{f^{\frac12}}(1-\Delta)^{-\frac{d}{2}}M_{f^{\frac12}}\Big\|_{\infty}=$$
$$=\Big\|M_{f^{\frac12}}(1-\Delta)^{-\frac{d}{2}}M_{f^{\frac12}}\Big|_{L_2([-1,1]^d)}\Big\|_{\infty}.$$
By Lemma \ref{non-compact to compact lemma}, we have
$$M_{f^{\frac12}}(1-\Delta)^{-\frac{d}{2}}M_{f^{\frac12}}\Big|_{L_2([-1,1]^d)}=$$
$$=M_{f^{\frac12}}(1-\Delta_{\mathbb{T}^d})^{-\frac{d}{2}}M_{f^{\frac12}}\Big|_{L_2([-1,1]^d)}+M_{f^{\frac12}}b(\nabla_{\mathbb{T}^d}) M_{f^{\frac12}}\Big|_{L_2([-1,1]^d)}.$$
By triangle inequality, we have 
$$\Big\|M_{f^{\frac12}}(1-\Delta_{\mathbb{T}^d})^{-\frac{d}{2}}M_{f^{\frac12}}\Big|_{L_2([-1,1]^d)}\Big\|_{\infty}\leq$$
$$\leq\Big\|M_{f^{\frac12}}(1-\Delta)^{-\frac{d}{2}}M_{f^{\frac12}}\Big|_{L_2([-1,1]^d)}\Big\|_{\infty}+\Big\|M_{f^{\frac12}}b(\mathbb{T}^d)M_{f^{\frac12}}\Big|_{L_2([-1,1]^d)}\Big\|_{\infty}=$$
$$=\Big\|M_{f^{\frac12}}(1-\Delta)^{-\frac{d}{2}}M_{f^{\frac12}}\Big\|_{\infty}+\Big\|M_{f^{\frac12}}b(\mathbb{T}^d)M_{f^{\frac12}}\Big\|_{\infty}\leq $$
$$\leq \Big\|(1-\Delta)^{-\frac{d}{4}}M_f(1-\Delta)^{-\frac{d}{4}}\Big\|_{\infty}+c_d\Big\|(1-\Delta_{\mathbb{T}^d})^{-\frac{d+1}{4}}M_f(1-\Delta_{\mathbb{T}^d})^{-\frac{d+1}{4}}\Big\|_{\infty}.$$
The assertion follows from Proposition \ref{space boundedness} and Lemma \ref{postcritical torus lemma}.
\end{proof}

The following proposition is a version of Proposition \ref{space boundedness optimal} for $\mathbb{T}^d.$

\begin{prop}\label{torus boundedness optimal} Let $d\in\mathbb{N}.$ Let $f=\mu(f)\in\mathtt{M}_{\psi}(0,1).$ We have
$$\Big\|(1-\Delta_{\mathbb{T}^d})^{-\frac{d}{4}}M_{f\circ r_d}(1-\Delta_{\mathbb{T}^d})^{-\frac{d}{4}}\Big\|_{\infty}\geq c_d\|f\|_{\mathtt{M}_{\psi}}.$$
\end{prop}
\begin{proof} Firstly, note that
$$\Big\|(1-\Delta)^{-\frac{d}{4}}M_{f\circ r_d}(1-\Delta)^{-\frac{d}{4}}\Big\|_{\infty}=\Big\|M_{f^{\frac12}\circ r_d}(1-\Delta)^{-\frac{d}{2}}M_{f^{\frac12}\circ r_d}\Big\|_{\infty}=$$
$$=\Big\|M_{f^{\frac12}\circ r_d}(1-\Delta)^{-\frac{d}{2}}M_{f^{\frac12}\circ r_d}\Big|_{L_2([-1,1]^d)}\Big\|_{\infty}.$$
By Lemma \ref{non-compact to compact lemma}, we have
$$M_{f^{\frac12}\circ r_d}(1-\Delta)^{-\frac{d}{2}}M_{f^{\frac12}\circ r_d}\Big|_{L_2([-1,1]^d)}=$$
$$=M_{f^{\frac12}\circ r_d}(1-\Delta_{\mathbb{T}^d})^{-\frac{d}{2}}M_{f^{\frac12}\circ r_d}\Big|_{L_2([-1,1]^d)}+M_{f^{\frac12}\circ r_d}b(\nabla_{\mathbb{T}^d}) M_{f^{\frac12}\circ r_d}\Big|_{L_2([-1,1]^d)}.$$
By triangle inequality, we have
$$\Big\|M_{f^{\frac12}\circ r_d}(1-\Delta)^{-\frac{d}{2}}M_{f^{\frac12}\circ r_d}\Big|_{L_2([-1,1]^d)}\Big\|_{\infty}\leq$$
$$\leq\Big\|M_{f^{\frac12}\circ r_d}(1-\Delta_{\mathbb{T}^d})^{-\frac{d}{2}}M_{f^{\frac12}\circ r_d}\Big|_{L_2([-1,1]^d)}\Big\|_{\infty}+$$
$$+\Big\|M_{f^{\frac12}\circ r_d}b(\mathbb{T}^d)M_{f^{\frac12}\circ r_d}\Big|_{L_2([-1,1]^d)}\Big\|_{\infty}=$$
$$=\Big\|M_{f^{\frac12}\circ r_d}(1-\Delta_{\mathbb{T}^d})^{-\frac{d}{2}}M_{f^{\frac12}\circ r_d}\Big\|_{\infty}+\Big\|M_{f^{\frac12}\circ r_d}b(\mathbb{T}^d)M_{f^{\frac12}\circ r_d}\Big\|_{\infty}.$$

Recall that 
$$|b(n)|\leq c_d(1+|n|^2)^{-\frac{d+1}{2}},\quad n\in\mathbb{Z}^d.$$
Thus,
$$\Big\|(1-\Delta)^{-\frac{d}{4}}M_{f\circ r_d}(1-\Delta)^{-\frac{d}{4}}\Big\|_{\infty}\leq$$
$$\leq\Big\|M_{f^{\frac12}\circ r_d}(1-\Delta_{\mathbb{T}^d})^{-\frac{d}{2}}M_{f^{\frac12}\circ r_d}\Big\|_{\infty}+\Big\|M_{f^{\frac12}\circ r_d}a(\mathbb{T}^d)M_{f^{\frac12}\circ r_d}\Big\|_{\infty}\leq$$
$$\leq\Big\|(1-\Delta_{\mathbb{T}^d})^{-\frac{d}{4}}M_{f\circ r_d}(1-\Delta_{\mathbb{T}^d})^{-\frac{d}{4}}\Big\|_{\infty}+c_d\Big\|(1-\Delta_{\mathbb{T}^d})^{-\frac{d+1}{4}}M_{f\circ r_d}(1-\Delta_{\mathbb{T}^d})^{-\frac{d+1}{4}}\Big\|_{\infty}.$$
By Proposition \ref{space boundedness optimal} and Lemma \ref{postcritical torus lemma}, we have
$$c_d'\|f\|_{\mathtt{M}_{\psi}}\leq\Big\|(1-\Delta_{\mathbb{T}^d})^{-\frac{d}{4}}M_{f\circ r_d}(1-\Delta_{\mathbb{T}^d})^{-\frac{d}{4}}\Big\|_{\infty}+c_d''\|f\|_1.$$

In other words, we have
\begin{equation}\label{tbo eq1}
\Big\|(1-\Delta_{\mathbb{T}^d})^{-\frac{d}{4}}M_{f\circ r_d}(1-\Delta_{\mathbb{T}^d})^{-\frac{d}{4}}\Big\|_{\infty}\geq c_d'\|f\|_{\mathtt{M}_{\psi}}-c_d''\|f\|_1.
\end{equation}

On the other hand, we have
$$\Big\|(1-\Delta_{\mathbb{T}^d})^{-\frac{d}{4}}M_{f\circ r_d}(1-\Delta_{\mathbb{T}^d})^{-\frac{d}{4}}\Big\|_{\infty}\geq\langle (1-\Delta_{\mathbb{T}^d})^{-\frac{d}{4}}M_{f\circ r_d}(1-\Delta_{\mathbb{T}^d})^{-\frac{d}{4}}1,1\rangle=$$
$$=\langle (1-\Delta_{\mathbb{T}^d})^{-\frac{d}{4}}M_{f\circ r_d}1,1\rangle=\langle M_{f\circ r_d}1,(1-\Delta_{\mathbb{T}^d})^{-\frac{d}{4}}1\rangle=\langle M_{f\circ r_d}1,1\rangle.$$
Thus,
\begin{equation}\label{tbo eq2}
\Big\|(1-\Delta_{\mathbb{T}^d})^{-\frac{d}{4}}M_{f\circ r_d}(1-\Delta_{\mathbb{T}^d})^{-\frac{d}{4}}\Big\|_{\infty}\geq \|f\|_1.
\end{equation}

Combining \eqref{tbo eq1} and \eqref{tbo eq2}, we obtain
$$(1+c_d'')\Big\|(1-\Delta_{\mathbb{T}^d})^{-\frac{d}{4}}M_{f\circ r_d}(1-\Delta_{\mathbb{T}^d})^{-\frac{d}{4}}\Big\|_{\infty}\geq c_d'\|f\|_{\mathtt{M}_{\psi}}.$$
This completes the proof.
\end{proof}

\subsection{Optimality of Cwikel-Solomyak estimates within the class of Orlicz spaces}\label{optimality-subsection}\label{optimality subsection}

The following material is standard; for more details we refer the reader to \cite{LSZ-book,Simon-book}.
Let $H$ be a complex separable infinite dimensional Hilbert space, and let $B(H)$ denote the set of all bounded operators on $H$, and let $K(H)$ denote the ideal of compact operators on $H.$ Given $T\in K(H),$ the sequence of singular values $\mu(T) = \{\mu(k,T)\}_{k=0}^\infty$ is defined as:
\begin{equation*}
\mu(k,T) = \inf\{\|T-R\|_{\infty}:\quad \mathrm{rank}(R) \leq k\}.
\end{equation*}
Let $p \in (0,\infty).$ The weak Schatten class $\mathcal{L}_{p,\infty}$ is the set of operators $T$ such that $\mu(T)$ is in the weak $L_p$-space $l_{p,\infty}$, with quasi-norm:
\begin{equation*}
\|T\|_{p,\infty} = \sup_{k\geq 0} (k+1)^{\frac1p}\mu(k,T) < \infty.
\end{equation*}
Obviously, $\mathcal{L}_{p,\infty}$ is an ideal in $B(H).$ 

Our next lemma follows from Proposition \ref{torus boundedness optimal} and provides one of the two ingredients in the proof of our third main result.

\begin{lem}\label{simple uniform lemma} Let $E(\mathbb{T}^d)$ be a symmetric Banach function space on $\mathbb{T}^d.$ Suppose that
$$\Big\|(1-\Delta_{\mathbb{T}^d})^{-\frac{d}{4}}M_f(1-\Delta_{\mathbb{T}^d})^{-\frac{d}{4}}\Big\|_{\infty}\leq c_{d,E}\|f\|_E,\quad f\in E(\mathbb{T}^d).$$
It follows that $E\subset \mathtt{M}_{\psi}.$
\end{lem}
\begin{proof} Take $h=\mu(h)\in E(0,1)$ and let $f=h\circ r_d.$ We have
$$\Big\|(1-\Delta_{\mathbb{T}^d})^{-\frac{d}{4}}M_{h\circ r_d}(1-\Delta_{\mathbb{T}^d})^{-\frac{d}{4}}\Big\|_{\infty}\leq c_{d,E}\omega_d\|h\|_E,\quad h=\mu(h)\in E(0,1).$$
By Proposition \ref{torus boundedness optimal}, we have
$$c_d\|h\|_{\mathtt{M}_{\psi}}\leq c_{d,E}\omega_d\|h\|_E,\quad h=\mu(h)\in E(0,1).$$
This completes the proof.
\end{proof}

The next lemma demonstrates efficiency of general theory of symmetric function spaces in the study of Cwikel estimates.

\begin{lem}\label{orlicz in marcinkiewicz} If $N$ is an Orlicz function such that $L_N(0,1)\subset\mathtt{M}_{\psi}(0,1),$ then $L_N(0,1)\subset L_M(0,1),$ where $M(t)=t\log(e+t),$ $t>0.$
\end{lem}
\begin{proof} We have
$$\|\chi_{(0,t)}\|_{\mathtt{M}_{\psi}}\leq c_N\|\chi_{(0,t)}\|_{L_N},\quad 0<t<1.$$
Thus,
$$t\log(\frac{e}{t})\leq \frac{c_N}{N^{-1}(\frac1t)},\quad t\in(0,1).$$
Setting $u=t^{-1},$ we write
$$u^{-1}\log(eu)\leq \frac{c_N}{N^{-1}(u)},\quad u>1.$$
Setting $v=N^{-1}(u),$ we write
$$N(v)^{-1}\log(eN(v))\leq \frac{c_N}{v},\quad v>N^{-1}(1).$$
Equivalently,
$$N(v)\geq c_N^{-1}v\log(eN(v)),\quad v>N^{-1}(1).$$
By convexity,
$$N(v)\geq \frac{v}{N^{-1}(1)},\quad v>N^{-1}(1).$$
Thus,
$$N(v)\geq c_N^{-1}v\log(\frac{ev}{N^{-1}(1)}),\quad v>N^{-1}(1).$$
Thus,
$$N(v)\geq c_N'M(v),\quad v>c_N''.$$
\end{proof}

The following corollary is our third main result. It demonstrates that Cwikel inequality proved by Solomyak (for even $d$) cannot be improved within the class of Orlicz spaces. Observe that a version of Solomyak inequality for an arbitrary $d$ is established in \cite{SZ-solomyak}.

It is interesting to compare the result of the following corollary with Theorem 9.4 in \cite{Shar14} which is proved under an artificial condition on Orlicz function $N.$ The result below holds for an arbitrary Orlicz function.

\begin{cor}\label{optimality within orlicz} If $N$ is an Orlicz function such that
$$\Big\|(1-\Delta_{\mathbb{T}^d})^{-\frac{d}{4}}M_f(1-\Delta_{\mathbb{T}^d})^{-\frac{d}{4}}\Big\|_{\infty}\leq c_{d,N}\|f\|_{L_N},\quad f\in L_N(\mathbb{T}^d),$$
then $L_N(0,1)\subset L_M(0,1).$ 
\end{cor}
\begin{proof} By Lemma \ref{simple uniform lemma}, $L_N\subset\mathtt{M}_{\psi}.$ By Lemma \ref{orlicz in marcinkiewicz}, $L_N\subset L_M.$
\end{proof}

\subsection{Cwikel-Solomyak estimate in Lorentz ideals}

We have the uniform norm estimate for Cwikel operator. Solomyak proved the $\|\cdot\|_{1,\infty}$-quasi-norm estimate for Cwikel operator. The next natural step is to involve real interpolation to obtain Lorentz norm estimates for Cwikel operator.

\begin{cor} Let $1<p<\infty$ and $1\leq q\leq \infty.$ We have
$$\Big\|(1-\Delta_{\mathbb{T}^d})^{-\frac{d}{4}}M_f(1-\Delta_{\mathbb{T}^d})^{-\frac{d}{4}}\Big\|_{p,q}\leq c_{p,q,d}\|f\|_{[\Lambda_{\phi},\mathtt{M}_{\psi}]_{1-\frac1p,q}},\quad f\in [\Lambda_{\phi},\mathtt{M}_{\psi}]_{1-\frac1p,q}(\mathbb{T}^d).$$	
\end{cor}
\begin{proof} Let $A:\mathtt{M}_{\psi}(\mathbb{T}^d)\to \mathcal{L}_{\infty}$ be a bounded operator defined by the setting
$$A:f\to (1-\Delta_{\mathbb{T}^d})^{-\frac{d}{4}}M_f(1-\Delta_{\mathbb{T}^d})^{-\frac{d}{4}}.$$
We have that $A:\Lambda_{\psi}(\mathbb{T}^d)\to\mathcal{L}_{1,\infty}$ is bounded. By real interpolation, we have
$$A:[\Lambda_{\phi},\mathtt{M}_{\psi}]_{1-\frac1p,q}(\mathbb{T}^d)\to [\mathcal{L}_{1,\infty},\mathcal{L}_{\infty}]_{1-\frac1p,q}.$$
It is well-known that
$$[\mathcal{L}_{1,\infty},\mathcal{L}_{\infty}]_{1-\frac1p,q}=\mathcal{L}_{p,q}.$$
This completes the proof.
\end{proof}

\end{document}